\documentclass[11pt]{article}
\usepackage{amssymb,latexsym,amsmath,amsbsy,amsthm,amsxtra,amsgen,dsfont,pdfsync,color}
\newfont{\msbm}{msbm10 at 11pt}

\newfont{\msbmsm}{msbm10 at 8pt}

\newtheorem{Theo}{Theorem}

\newtheorem{Prop}[Theo]{Proposition}

\begin{document}
\title{On the time to absorption in $\Lambda$-coalescents}

\author{G\"otz Kersting\thanks{Institut f\"ur Mathematik, Goethe Universit\"at, Frankfurt am Main, Germany \newline kersting@math.uni-frankfurt.de, wakolbinger@math.uni-frankfurt.de \newline Work partially supported by the DFG Priority Programme SPP 1590 ``Probabilistic Structures in Evolution''}$\ $ and Anton Wakolbinger$^*$}
\date{}
\maketitle
\begin{abstract}
We present a law of large numbers and a central limit theorem for the time to absorption of $\Lambda$-coalescents, started from $n$ blocks, as $n \to \infty$. The proofs rely on an approximation of the logarithm of the block-counting process of $\Lambda$-coalescents with a dust component by means of a drifted subordinator.\\
{\em AMS 2010 subject classification:} 60J75 (primary), 60J27, 60F05 (secondary)$^{\color{white} \big|}$\\
{\em Keywords:} coalescents, time to absorption, law of large numbers, central limit
theorem, subordinator with drift
\end{abstract}

\section{Introduction and main results}
How long does it take for the ancestral lineages of a large sample of individuals back to its common ancestor? 
For population of constant size this turns into a question on the absorption time of a coalescents, which describes the genealogical tree of $n$ individuals by means of merging partitions. Here we consider coalescent with multiple mergers, also known as $\Lambda$-coalescents, which were introduced in 1999 by Pitman \cite{Pi} and Sagitov \cite{Sa}. If $\Lambda$ is a finite, non-zero measure on $[0,1]$, then the $\Lambda$-coalescent started with $n$ blocks is a continuous-time Markov chain $(\Pi_n(t), t \geq 0)$ taking its values in the set of partitions of $\{1, \dots, n\}$.  It has the property that whenever there are $b$ blocks, each possible transition that involves merging $k \geq 2$ of the blocks into a single block happens at rate
\[
\lambda_{b,k} = \int_{[0,1]} p^{k} (1-p)^{b-k} \: \frac{\Lambda(dp)}{p^2}\ ,
\]
and these are the only possible transitions.  
Let $N_n(t)$ be the number of blocks in the partition $\Pi_n(t)$, $t \ge 0$.  Then 
\[\tau_n := \inf\{t\ge 0: N_n(t) = 1\}\] is the time of the last merger, also called the {\em absorption time} of the coalescent started in $n$ blocks.   We will investigate the asymptotic distribution of $\tau_n$ as $n \to \infty$.

Our first result is a law of large numbers for the  times $\tau_n$. Let 
\[ \mu:=\int_{[0,1]} \log \frac 1{1-p} \, \frac {\Lambda(dp)}{p^2} \ , \]
in particular $\mu=\infty$ in case of $\Lambda(\{0\})>0$ or $\Lambda(\{1\})>0$.

\begin{Theo}\label{LLN} For any $\Lambda$-coalescent,
\begin{align} \frac{\tau_n}{\log n} \to \frac 1\mu \ . \label{lln}
\end{align}
in probability as $n \to \infty$.
\end{Theo}

This theorem says that in a $\Lambda$-coalescent the number of blocks decays at least at an exponential rate. If $\mu=\infty$, then the right-hand limit is 0, and the coalescent  decreases even super-exponentially fast. The case $\mu<\infty$ is equivalently captured by the simultaneous validity of the  conditions 
\[ \int_{[0,1]}\frac{\Lambda(dp)}p < \infty \quad \text{ and }\quad \int_{[0,1]} \log \frac 1{1-p} \, \Lambda (dp) < \infty \ . \]
The first one is a requirement on $\Lambda$ in the vicinity of 0, it prohibits  a swarm of small mergers (as they occur in coalescents coming down from infinity, meaning that the $\tau_n$ are bounded in probability uniformly in $n$). The second is a condition on $\Lambda$ in the vicinity of 1. It rules out the possibility of mergers which, although appearing only every now and then,  are so vast that they make the coalescent collapse. -- A counterpart to Theorem \ref{LLN}, with $\tau_n$  in \eqref{lln} replaced by its expectation, was obtained by Herriger and M\"ohle \cite{HeMoe}.

Our second result is a central limit theorem. Here we confine ourselves to coalescents with $\mu<\infty$. Then the function 
\begin{equation} \label{ourf} f(y):= \int_{[0,1]} \frac{1-(1-p)^{e^y} }{e^y} \, \frac{\Lambda(dp)}{p^2} \ , \ y \in \mathbb R 
\end{equation} 
is everywhere finite. Also $f$ is a positive, monotone decreasing, continuous function with the property $f(y)\to 0$ for $y \to \infty$. Let
\[ b_n := \int_\kappa^{\log n} \frac{dy}{\mu- f(y)} \ , \]
where we choose $\kappa \ge 0$ such that
\[ f(y) \le \frac \mu 2 \text{ for all } y \ge \kappa \ . \]

\begin{Theo}\label{CLT} Assume that $\mu< \infty$ and moreover
\[ \sigma^2 := \int_{[0,1]} \Big(\log \frac 1{1-p}\Big)^2\, \frac{\Lambda(dp)}{p^2} < \infty \ . \]
Then
\begin{align}\label{celith}
 \frac{\tau_n - b_n}{\sqrt{\log n}}\ \stackrel d \to\ N(0, \sigma^2/\mu^3) 
\end{align}
 as $n \to \infty$.
\end{Theo}

\newpage 

Under the additional condition 
\begin{align}\label{integral}
\int_{[0,1]} \log \frac 1p \, \frac{\Lambda(dp)}{p} < \infty \ .
\end{align}
the CLT \eqref{celith} has been obtained by Gnedin, Iksanov and Marynych \cite{Gne}, with $b_n $ replaced by $\log n/\mu$. (Their condition (9) is equivalent to the above condition \eqref{integral}, see  Remark 13 in \cite{Ke}).  Thus the question arises, whether the simplified centering by $\log n/\mu$ is always feasible. The next proposition shows that this can be done under a condition that is weaker than \eqref{integral}, but not in any case.

\begin{Prop} \label{Prop2}
Let $0\le c < \infty$. Then 
\begin{align} \label{condi2} b_n= \frac{\log n}{\mu}+\frac {2c} {\mu^2}\sqrt {\log n} + o(\sqrt {\log n})
\end{align}
as $n \to \infty$, if and only if
\begin{align}\label{condi} \sqrt{\log \tfrac 1r}\int_{[0,r]} \frac{\Lambda(dp)}p \to c \end{align}
as $r\to 0$.
\end{Prop}

\paragraph{Example.} We consider for $\gamma \in \mathbb R$ the finite measures
\[ \Lambda(dp) = \big(1+\log \tfrac 1p\big)^{-\gamma} \, dp\ , \ 0\le p \le 1 \ . \]
For $\gamma=0$ this gives the Bolthausen-Sznitman coalescent. For $\gamma >1$ it leads to  coalescents with $\mu, \sigma^2<\infty$. Note that \eqref{integral} is satisfied iff $\gamma >2$, and \eqref{condi} is fulfilled iff $\gamma >3/2$. Thus within the range $1<\gamma \le 3/2$ one has to come back to the constants   $b_n$  in the central limit theorem.

The law of large numbers from Theorem 1 holds for all $\gamma >1$. For the  regime $\gamma \le 1$, Theorem~\ref{LLN} just tells us that  $\tau_n= o_P(\log n)$. For $\gamma=0$, the Bolthausen-Sznitman coalescent, it is known that $\tau_n$ is already down to the order $\log\log n$ \cite{Go}.  For $\gamma <0$, applying Schweinsberg's criterion \cite{Schw},  it can be shown that the coalescents come down from infinity. There remains the gap  $0< \gamma \le 1$. It is tempting to conjecture that $\tau_n$ is of order $(\log n)^\gamma$ for $0<\gamma < 1$.
\qed

\mbox{}\\
If equation \eqref{condi} is violated then the subsequent approximation to $b_n$ may be practical.
Starting from the  identity  
\[ \frac{1}{\mu-f(y)} = \frac 1\mu +\frac {f(y)}{\mu^2} + \frac{f^2(y)}{\mu^3} +\cdots + \frac{f^k(y)}{\mu^{k+1}}+ \frac {f^{k+1}(y)}{\mu^{k+1}(\mu-f(y))}\]
we obtain the expansion
\begin{align*} b_n = \frac {\log n}\mu + \frac {1}{\mu^2} \int_0^{\log n} f(y)\, dy + \cdots + \frac 1{\mu^{k+1}}\int^{\log n}_0 f^k(y)\, dy + O\Big( \int_0^{\log n} f^{k+1}(y)\, dy\Big) \ .
\end{align*}

\newpage

Let us now explain the method of proving Theorems \ref{LLN} and \ref{CLT}. We are mainly dealing with $\Lambda$-coalescents having a {\em dust component}. Shortly speaking these are the coalescents for which the rate, at which a single lineage merges with some others  from the sample, stays bounded as the sample size tends to infinity. As is well-known this property is characterized by the condition
\begin{align}
\label{dustcond} \int_{[0,1]} \frac{\Lambda(dp)}{p} < \infty \ . 
\end{align}
An established tool for the analysis of a $\Lambda$-coalescent with dust is the subordinator $S=(S_t)_{t \ge 0}$, which is used to approximate the logarithm of its block-counting process $N_n=(N_n(t))_{t \ge 0}$ (see e.g. Pitman \cite{Pi}, M\"ohle \cite{Moe}, and the above mentioned paper by Gnedin et al \cite{Gne}). We will recall this subordinator in Sec.~3. Indeed,  analogues of Theorems 1 and 2 are well-known for first-passage times of subordinators with finite first resp. second moment. However, this approximation neglects the subtlety that a coalescent of $b$ lineages results in a downward jump of size $b-1$ (and not $b$) for the process $N_n$. This effect becomes significant when many small jumps accumulate over time, as it happens close to the dustless case (and as it becomes visible in Proposition \ref{Prop2} and in the above example). Then the appropriate approximation is provided by a {\em drifted} subordinator $Y_n=(Y_n(t))_{t \ge0}$,   given by the SDE
\[ Y_n(t) = \log n -S_t + \int_0^t f(Y_n(s))\, ds \ , \ t \ge 0\ ,\]
with initial value $Y_n(0)=\log n$. The drift compensates the just mentioned difference between $b$ and $b-1$. In Kersting et al \cite{Ke} it is shown that  
\[  \sup_{t<\tau_n}\big|Y_n(t)- \log N_n(t)\big|=O_P(1)\]
as $n \to \infty$, that is, these random variables are bounded in probability. In Sec.~3 we suitably strengthen this result. In Sec.~2 we provide the required limit theorems for passage times for a  more general class of drifted subordinators. The above results are then proved in Sec.~4.

It turns out that the regime considered by Gnedin et al \cite{Gne} is the one in which the random variables $\int_0^{\tau_n} f(Y_n(s))\, ds$ are bounded in probability uniformly in $n$. This can be seen to be equivalent to  the requirement $\int_0^\infty f(y)\, dy<\infty$, which likewise is equivalent to \eqref{integral} (see the proof of Corollary 12 in \cite{Ke}). Under this assumption Gnedin et al \cite{Gne} proved their central limit theorem  also with non-normal   (stable or  Mittag-Leffler) limiting distributions of $\tau_n$. A similar generalization of Theorem \ref{CLT} is feasible in the general dust case, without the requirement \eqref{integral}.

\section{Limit theorems for a drifted subordinator}

Let $S=(S_t)_{t \ge 0}$ be a pure jump subordinator with L\'evy measure $\lambda$ on $(0,\infty)$. Recall that this requires
\[ \int_0^\infty (y\wedge 1) \, \lambda(dy) < \infty \ .\]
With regard to the mentioned properties of the function in \eqref{ourf}, let $f:\mathbb R \to \mathbb R$ be an arbitrary positive, non-increasing, continuous function with  
\[ \lim_{y\to \infty} f(y)=0\ . \]
Let the process $Y^z=(Y^z_t)_{t \ge 0}$ denote the unique solution of the SDE
\begin{align} \label{SDE} Y^z_t = z-S_t + \int_0^t f(Y^z_s)\, ds \end{align}
with initial value $z> 0$. 
We will investigate the asymptotic behaviour of its passage  times across $x \in \mathbb R$,
\[ T^z_x := \inf\{ t \ge 0: Y^z_t < x \}\ ,\]
 in the limit $z \to \infty$. 

The first result provides a law of large numbers. Denote
\begin{align} \mu := \int_{(0,\infty)} y\, \lambda(dy) \ . 
\label{defmu}
\end{align}
\begin{Prop} \label{Prop3} Assume that $\mu<\infty$. Then for any $x \in \mathbb R$
\[ \frac{1}z \, T^z_x \to \frac 1\mu \]
in probability as $z\to \infty$.
\end{Prop}

\begin{proof} 
Let $z >x$. Then
\[ \{T^z_x \ge t\} = \{ Y^z_s \ge x \text{ for all }s \le t \} = \Big\{S_s \le z-x + \int_0^sf(Y^z_u)\, du \text{ for all } s \le t\Big\} \ . \]
By positivity of the function $f$ it follows
$ \mathbf P( T^z_x \ge t) \ge \mathbf P( S_t \le  z-x) $, thus for any $\varepsilon >0$ 
\begin{equation}\label{oneplusepsilon}
\mathbf P\Big(T_x^{  z}\ge  (1-\varepsilon)\frac z{\mu} \Big) \ge \mathbf P\big( S_{(1-\varepsilon)z/\mu} \le z -x\big) \ .
\end{equation}
Now $\mu=\mathbf E[S_1]$, thus by the law of large numbers
\[\frac{S_t}t \to \mu \]
a.s., hence the right-hand term in \eqref{oneplusepsilon} converges to 1 for $z \to \infty$ and also 
\[\mathbf P\Big(T_x^{  z}\ge  (1-\varepsilon)\frac z{\mu} \Big) \to 1 \ . \]

On the other hand,
\begin{align*}\{ T^z_x \ge t\} = \{ Y^z_s \ge x \text{ for all } s \le t\}=\Big\{ Y^z_s \ge x \text{ for all }s \le t\, ,\,  S_t \le z -x + \int_0^t f(Y_s^z)\,ds \Big\} \ .
\end{align*}
Monotonicity of $f$ implies
$ \mathbf P(  T^z_x \ge t ) \le \mathbf P \big(S_t \le z -x + t f(x) \} \big) $.
Therefore, since $f(x)\to 0$ as $x \to \infty$,
\begin{align*} \mathbf P\Big(T_x^{ z}\ge (1+\varepsilon)\frac z{\mu} \Big) &\le \mathbf P\Big( S_{(1+\varepsilon)z/\mu } \le z  -x + (1+\varepsilon)\frac z{\mu}f(x) \Big)\\&\le  \mathbf P\big( S_{(1+\varepsilon)z/\mu } \le z(1+\varepsilon/2) -x\big) \end{align*}
if only $x$ is sufficiently large. Now the right-hand term converges to 0, thus it follows that
\[ \mathbf P\Big(T_x^{ z}\ge (1+\varepsilon)\frac z{\mu} \Big) \to 0\ . \]
Note that we proved this result only for $x$ sufficiently large, depending on $\varepsilon$. However, this restriction may be skipped, since for fixed $x_1<x_2$ the random variables $T_{x_1}^z-T_{x_2}^z$ are  bounded in probability uniformly in $z$. Thus altogether we have for any $x$
\[ \mathbf P\Big( (1-\varepsilon) \frac z{\mu} \le T^z_x < (1+\varepsilon)  \frac z{\mu} \Big)\to 1\]
as $z \to \infty$, which (since $\varepsilon > 0$ was arbitrary) is our assertion.
\end{proof}

\mbox{}\\
Now we turn to a central limit theorem for  passage times of the processes $Y^z$.  
Let the function $\beta_z$, $z\ge \kappa$, be given by
\begin{equation}\label{defbeta}
 \beta_z:= \int_\kappa^z \frac{dy}{\mu-f(y)} \ , 
 \end{equation}
where we choose $\kappa\ge 0$ so large that
\[ \sup_{y \ge \kappa} f(y) \le \frac \mu 2 \ . \]

\begin{Prop}\label{Prop4} Suppose that
\begin{align}\label{sigma} \sigma^2 := \int_{(0,\infty)} y^2 \, \lambda (dy) < \infty\ . 
\end{align}
 Then
\[ \frac{ T_x^z- \beta_z}{\sqrt z}\ \stackrel d \to\ N(0, \sigma^2/\mu^3) \]
as $z \to \infty$.
\end{Prop}

\begin{proof} (i) Note again that for $x_1 < x_2$ the random variables $T^z_{x_1}-T^z_{x_2}$ are bounded in probability uniformly in $z$. Thus it suffices to prove our theorem for all $x \ge x_0$ for some $x_0\in \mathbb R$. Therefore we may    change $f(x)$ for all $x<x_0$. We do it in such a way that $f(x)\le \mu/2$ for all $x\in \mathbb R$, without touching the other properties of $f$. Thus we assume from now that
\begin{align} \label{fmu} f(y) \le \frac \mu 2 \quad \text{ for all } y \in \mathbb R\ 
\end{align}
and set $\kappa=0\ $
in \eqref{defbeta}. Consequently, 
\begin{equation}\label{betabounds}
\frac z\mu \le \beta_z\le \frac{2z}\mu\ , \quad z >0\ .
\end{equation}

For any $z >0$ we define the function $\rho^z(t)=\rho^z_t$, $0 \le t \le \beta_z$, such that
\[ \beta_{\rho^z(t)}=\beta_z-t\ , 0\le t \le \beta_z\ , \]
in particular $\rho^z(0)=z$ and $\rho^z(\beta_z)=0$. This means that $\rho^z$ arises by first inverting the function $\beta$ (restricted to the interval $[0, z]$), and then reversing the time parameter on its domain $[0, \beta_z]$. By differentiation we obtain
\[ \dot\rho^z_t= f(\rho^z_t)-\mu\ , \]
  consequently $\dot \rho_t \le - \mu/2$ and
\[ \rho^z_t = z - \mu t + \int_0^t f(\rho^z_s)\, ds\ . \]

(ii) A glimpse on \eqref{SDE} suggests that $\rho^z$ will make a good approximation for the process $Y^z$.
In order to estimate their difference observe that
\[ Y_t^z- \rho^z_t = -(S_t-\mu t) + \int_0^t (f(Y^z_s)-f(\rho^z_s))\, ds \ .\] 
For given $t>0$ define
\begin{align*}
u_t =\begin{cases} \sup \{ s<t: Y^z_s \le \rho^z_s\} \mbox{ on the event } Y_t^z > \rho^z_t 
\\
\sup \{ s<t: Y^z_s \ge \rho^z_s\} \mbox{ on the event } Y_t^z < \rho^z_t
\end{cases}
\end{align*}
and $u_t:=t$ on the event $Y_t^z = \rho^z_t$.
We have $0 \le u_t \le t$, since $Y^z_0=z= \rho_0^z$. Because $f$ is a decreasing function, the event $Y_t^z > \rho^z_t$ implies that
\begin{align*}Y_t^z- \rho^z_t &\le Y_t^z- \rho^z_t - \int_{u_t}^t (f(Y^z_s)-f(\rho^z_s))\, ds - (Y^z_{u_t-}- \rho^z_{u_t-})\\ 
&= -(S_t- \mu t) + (S_{u_t-}-\mu u_t) \ .
\end{align*}
On the event $Y_t^z < \rho^z_t$ there is an analogous estimate from below, altogether
\begin{align*}
 |Y_t^z- \rho^z_t| \le 2 M_t \quad \text{ with } M_t:=  \sup_{u \le t} |S_u-\mu u| \ . 
 \end{align*}
Consequently, $Y^z_s \ge \rho^z_s-2M_s \ge \rho^z_s-2M_t$ for $s \le t$ and by means of the monotonicity of $f$
\begin{align*}
\int_0^t f(Y^z_s)\, ds -\int_0^tf(\rho^z_s)\, ds \le \int_0^t f(\rho^z_s-2M_t)\, ds -\int_0^tf(\rho^z_s)\, ds \le2M_t f(\rho^z_t-2M_t)\ .
\end{align*}
An analoguous estimate is valid from below and we obtain
\begin{align}
\Big|\int_0^t f(Y^z_s)\, ds -\int_0^tf(\rho^z_s)\, ds \Big| \le 2M_t f(\rho^z_t-2M_t)\ .
\label{Mrho}
\end{align}

At this point we recall that under the above assumptions on the subordinator $S$ by Donsker's invariance principle we have
\[ M_t =  O_P( \sqrt t) \]
as $t \to \infty$.

(iii) Now we derive some upper estimates of probabilities. Given $a,x\in \mathbb R$,  we have for any $c>0$
\begin{align*} \mathbf P&( T^z_x \ge \beta_z + a \sqrt z) = \mathbf P(Y^z_t \ge x \text{ for all } t \le \beta_z+a\sqrt z)\\ 
&= \mathbf P\Big(S_{\beta_z + a\sqrt z} \le z-x + \int_0^{\beta_z + a\sqrt z} f(Y^z_s) \, ds\, , \,Y^z_t \ge x \text{ for all } t \le \beta_z+a\sqrt z\Big)\\
&\le \mathbf P\Big(S_{\beta_z + a\sqrt z} \le z-x + f(x)(c+|a|)\sqrt z+ \int_0^{\beta_z -c\sqrt z} f(Y^z_s) \, ds\Big)
\end{align*}
We now bring \eqref{Mrho} into play. From the definition of $\rho^z$ we have, writing $\beta(y) = \beta_y$, that
\[\beta(\rho^z(\beta_z-c\sqrt z))=  c\sqrt z,\]
thus because of \eqref{betabounds}
\[\rho^z(\beta_z-c\sqrt z) \ge \frac{c\sqrt z}{2\mu}.\] 
Then  
on the event $M_{\beta_z} \le  {c\sqrt z} /{(8\mu)}$  we have
\[\rho^z(\beta_z-\sqrt z) - 2M_{\beta_z-c\sqrt z}\ge \frac{c\sqrt z}{2\mu}- \frac{c\sqrt z}{4\mu} = \frac{c\sqrt z}{4\mu}. \] 
Consequently, by means of \eqref{Mrho} and since $\beta_z \le 2z/\mu$
\begin{align}
\mathbf P( T^z_x &\ge \beta_z + a \sqrt z) \le \mathbf P\Big(M_{2z/\mu} >  \frac{c\sqrt z}{8\mu}\Big) \notag\\
&  \mbox{}+
\mathbf P\Big(S_{\beta_z + a\sqrt z} \le z-x + f(x)(c+|a|)\sqrt z+ \int_0^{\beta_z} f(\rho^z_s) \, ds+ \frac{c\sqrt z}{4\mu} f\Big(\frac{c\sqrt z}{4\mu}\Big) \Big)  \ . \label{Ungl1}
\end{align}
Moreover, by definition of $\rho^z$,
\[ z+\int_0^{\beta_z} f(\rho^z_s) \, ds =\rho^z(\beta_z)+\mu \beta_z =  \mu \beta_z . \]
Therefore, if we fix $\varepsilon>0$, let $c$ be so large that the first right-hand probability in \eqref{Ungl1} is smaller than $\varepsilon$, then choose $z$ so large that $(c/4\mu) f(\frac{c\sqrt z}{ 4\mu}) \le \varepsilon $, and also choose $x >0$ and  so large that $cf(x)(c+|a|) \le \varepsilon $, then we end up with
\[ \mathbf P( T^z_x \ge \beta_z + a \sqrt z) \le  \varepsilon  + \mathbf P\Big(S_{\beta_z + a\sqrt z} \le  \mu\beta_z  +2 \varepsilon\sqrt z   \Big)\ . \]
Also by the law of large numbers 
\[ S_{\beta_z + a\sqrt z}- S_{\beta_z} \sim \mu a\sqrt z \]
in probability. Therefore 
\[\mathbf P( T^z_x \ge \beta_z + a \sqrt z) \le  2\varepsilon+ \mathbf P\big(S_{\beta_z} \le  \mu\beta_z  + (-\mu a+3\varepsilon)\sqrt z   \big)\ .\]
Moreover $\mu \beta_z \sim z $, hence  
\[\mathbf P( T^z_x \ge \beta_z + a \sqrt z) \le 2 \varepsilon+ \mathbf P\big(S_{\beta_z} \le  \mu\beta_z  + (-\mu a+4\varepsilon)\mu^{1/2} \sqrt{\beta_z}  \big)\]
for large $z$. Now from assumption \eqref{sigma} and the central limit theorem there follows 
\[ \frac{S_t-\mu t} {\sqrt {\sigma^2  t}} \ \stackrel d\to\ L \ , \]
where   $L$ denotes a standard normal random variable. Thus 
\[ \limsup_{z \to \infty} \mathbf P( T^z_x \ge \beta_z + a \sqrt z)\le 2\varepsilon + \mathbf P (L \le (-\mu a+4\varepsilon)\mu^{1/2} \sigma^{-1})\ . \]
Note that the choice of $x$ depends on $\varepsilon$ in our proof. However, since again the differences $T^z_{x_1}-T^z_{x_2}$ are  bounded in probability uniformly  in $z$, this estimate generalizes to all $x$. Now letting $\varepsilon \to 0$ we obtain
\[\limsup_{z \to \infty} \mathbf P \Big( \frac{T^z_x- \beta_z}{\sqrt z} \ge a\Big) \le \mathbf P(L \le -\mu^{3/2} \sigma^{-1}a) \ . \]
This is the first part of our claim.

(iv) For the lower estimates we first  introduce the random variable
\[ R_{z,x} := \sup\{ t\ge 0: Y^z_t \ge x\} - \inf \{ t \ge 0: Y^z_t < x\} \]
which is the length of the time interval where $Y^z_t-x$ is changing from positive sign to ultimately negative sign (note that the paths of $Y^z$ are {\em not} monotone). We claim that these random variables are bounded in probability, uniformly in $z$ and $x$. Indeed, with 
\[ \eta_{z,x}:= \inf \{t \ge 0: Y^z_t < x\} \]
we have for $t>\eta=\eta_{z,x} $ because of $Y^z_\eta \le x$ and \eqref{fmu}
\begin{align*}Y_t^z &= Y_\eta^z -(S_t-S_\eta) +\int_\eta^tf(Y^z_s)\, ds\le x -(S_t-S_\eta) + \frac \mu 2 (t-\eta) \ .
\end{align*}
Thus $R_{z,x}$ is bounded from above by
\[ R_{z,x}' := \sup \{u \ge 0: (S_{\eta_{z,x}+u}-S_{\eta_{z,x}}) -  \mu u/2\le 0 \} \]
These random variables are a.s. finite. Moreover, they are identically distributed, since $\eta_{z,x}$ are stopping times. This proves that the $R_{z,x}$ are uniformly bounded in probability.

Now for the lower bounds we have for $a,b \in \mathbb R$
\begin{align*}
\mathbf P(T^z_x \ge \beta_z + a \sqrt z) &\ge \mathbf P( Y^z_t \ge x \text{ for all } t \le \beta_z+a\sqrt z\, ,\, R_{z,x} \le b)\\
&= \mathbf P( Y^z_t \ge x \text{ for all } \beta_z +a\sqrt z- b\le  t \le \beta_z+a\sqrt z\, ,\, R_{z,x} \le b)\ .
\end{align*}
For these $t$ we have 
\[Y_t^z = z- S_t+ \int_0^t f(Y^z_s)\, ds \ge z-S_{\beta_z+a\sqrt z}+ \int_0^{\beta_z +a\sqrt z- b} f(Y^z_s)\, ds \ , \]
therefore
\begin{align*}
\mathbf P(T^z_x \ge \beta_z + a \sqrt z) &\ge \mathbf P\Big( S_{\beta_z+a\sqrt z} \le z- x+ \int_0^{\beta_z+a\sqrt z- b}f(Y^z_s)\, ds \, , \, R_{z,x} \le b\Big)\\
&\ge \mathbf P\Big( S_{\beta_z+a\sqrt z} \le z- x+ \int_0^{\beta_z-c\sqrt z}f(Y^z_s)\, ds\Big) - \mathbf P(R_{z,x}>b)
\end{align*}
for $c$ sufficiently large.

We now bring, as in part (iii), \eqref{Mrho} into play. Proceeding analogously we obtain instead of \eqref{Ungl1} the estimate
\begin{align*}
\mathbf P( T^z_x \ge \beta_z + a \sqrt z) \ge &- \mathbf P(R_{z,x}>b)-\mathbf P\Big(M_{2z/\mu} >  \frac{c\sqrt z}{8\mu}\Big) \\
&  \mbox{}+
\mathbf P\Big(S_{\beta_z + a\sqrt z} \le z-x +  \int_0^{\beta_z-c\sqrt z} f(\rho^z_s) \, ds- \frac{c\sqrt z}{4\mu} f\Big(\frac{c\sqrt z}{4\mu}\Big) \Big)  \ . 
\end{align*}
Also, since $\rho^z_{\beta_z}=0$ and $\dot \rho^z_t \le -\mu/2$, 
\[ \int_{\beta_z-c\sqrt z}^{\beta_z} f(\rho^z_s) \, ds \le \int_0^{c\sqrt z} f(\mu s/2) \, ds  = o(\sqrt z) \ .\]
Hence, for given $\varepsilon>0$ and $z$ sufficiently large
\begin{align*}
\mathbf P( T^z_x \ge \beta_z + a \sqrt z) \ge &- \mathbf P(R_{z,x}>b)-\mathbf P\Big(M_{2z/\mu} >  \frac{c\sqrt z}{8\mu}\Big) \\
&  \mbox{}+
\mathbf P\Big(S_{\beta_z + a\sqrt z} \le z-\varepsilon \sqrt z +  \int_0^{\beta_z} f(\rho^z_s) \, ds- \frac{c\sqrt z}{4\mu} f\Big(\frac{c\sqrt z}{4\mu}\Big) \Big)  \ . 
\end{align*}
Returning to the  arguments of part (iii) we choose $b$, $c$ and then $z$ so large that we arrive at 
\[ \mathbf P( T^z_x \ge \beta_z + a \sqrt z)  \ge - 2\varepsilon + \mathbf P\Big(S_{\beta_z + a\sqrt z} \le \mu \beta_z- 2\varepsilon \sqrt z \Big)
\]
and further at
\[ \liminf_{z \to \infty} \mathbf P( T^z_x \ge \beta_z + a \sqrt z)\ge -3\varepsilon + \mathbf P (L \le (-\mu a-3\varepsilon)\mu^{1/2}\sigma^{-1} )\ . \]
The limit $\varepsilon \to 0$ leads to the desired lower estimate.
\end{proof}

\section{Approximating the block counting process}
In this section we derive a strengthening of a result in Kersting, Schweinsberg and Wakolbinger ~\cite{Ke} on the approximation to the logarithm of the block counting processes in the dust case. To this end, let us quickly recall the Poisson point process construction of the $\Lambda$-coalescent given in \cite{Ke}, which is a slight variation of the  construction provided by Pitman in \cite{Pi}.

This construction requires $\Lambda(\{0\}) = 0$, which is fulfilled for coalescents with dust.  Consider  a Poisson point process $\Psi$  on $(0, \infty) \times (0, 1] \times [0, 1]^n$ with intensity $$dt \times p^{-2} \Lambda(dp) \times du_1 \times \dots \times du_n\ ,$$  and let $\Pi_n(0) = \{\{1\}, \dots, \{n\}\}$ be the partition of the integers $1, \dots, n$ into singletons.  Suppose $(t, p, u_1, \dots, u_n)$ is a point of $\Psi$, and $\Pi_n(t-)$ consists of the blocks $B_1, \dots, B_b$, ranked in order by their smallest element.  Then $\Pi_n(t)$ is obtained from $\Pi_n(t-)$ by merging together all of the blocks $B_i$ for which $u_i \leq p$ into a single block.  These are the only times that mergers occur.  This construction is well-defined because almost surely for any fixed $t' < \infty$, there are only finitely many points $(t, p, u_1, \dots, u_n)$ of $\Psi$ for which $t \leq t'$ and at least two of $u_1, \dots, u_n$ are less than or equal to $p$.  The resulting process $\Pi_n = (\Pi_n(t), t \geq 0)$ is the $\Lambda$-coalescent.  When $(t, p, u_1, \dots, u_n)$ is a point of $\Psi$, we say that a $p$-merger occurs at time $t$.  

Condition \eqref{dustcond} allows us to approximate the number of blocks in the $\Lambda$-coalescent by a subordinator.  Let $\phi: (0, \infty) \times (0, 1] \times [0, 1]^n \rightarrow (0, \infty) \times (0, \infty]$ be the function defined by $$\phi(t, p, u_1, \dots, u_n) = (t, -\log(1-p)).$$  Now $\phi(\Psi)$ is a Poisson point process, and we can define a pure jump subordinator $(S(t), t \geq 0)$ having the property that $S(0) = 0$ and, if $(t, x)$ is a point of $\phi(\Psi)$, then $S(t) = S(t-) + x$. With $\lambda$ the L\'evy measure of $S$, the formulas \eqref{defmu} and \eqref{sigma} now read
\[  \mu=\int_{[0,1]} \log \frac 1{1-p} \, \frac {\Lambda(dp)}{p^2} \ \text{ and }\ \sigma^2 = \int_{[0,1]} \Big(\log \frac 1{1-p}\Big)^2\, \frac{\Lambda(dp)}{p^2} \ .\]
 This subordinator first appeared in the work of Pitman \cite{Pi} and was used to approximate the block-counting process by  Gnedin et al. \cite{Gne} and M\"ohle \cite{Moe}; the benefits of a refined approximation by a {\em drifted} subordinator were discovered in \cite{Ke}. We recall that the drift appears because a merging of $b$ out of $N_n(t)$ lines results in a decrease by $b-1$ and not  by $b$ lines, see equation (23) in \cite{Ke} for an  explanation of the form of the drift. 
 The next result 
provides a refinement of Theorem 10 in \cite{Ke}. 

\begin{Prop}\label{Prop5} Let
\[ \int_{[0,1]} \frac{\Lambda(dp)}p < \infty \ , \]
let $f$ be as in \eqref{ourf}, and let $Y_n$ be the solution of \eqref{SDE} with $z:= \log n$. 
Then for any $\varepsilon >0$ there is an $\ell <\infty$ such that
\[ \mathbf P\big( \sup_{t<\tau_n}| \log N_n(t)-Y_n(t)| \le \ell \, , \, Y_n(\tau_n) < \ell \big) \ge 1-\varepsilon \ . \]
\end{Prop}

\begin{proof} From \cite{Ke} we know that for given $\varepsilon >0$ there is an $r <\infty$ such that
\begin{align*}
\mathbf P\big( \sup_{t<\tau_n}| \log N_n(t)-Y_n(t)| \le r \big) \ge 1-\varepsilon/2 \  .
\end{align*}
Now we consider the size $\Delta_n$ of the last jump. Letting $(u_i,p_i)$, $i \ge 1$, be the points of the underlying Poisson point process with intensity measure $dt\, \Lambda(dp)/p^2$, the associated subordinator $S$ has jumps of size $v_i=-\log(1-p_i)$ at times $t_i$. Thus for any $c>0$ we have
\begin{align*}
\{ \Delta_n \le \log N_n(\tau_n-)-c\} &= \{ \tau_n=t_i \text{ and } -\log(1-p_i) \le \log N_n(t_i-)-c \text{ for some } i\ge 1\}\\
&= \Big\{ \tau_n=t_i \text{ and } p_i \le 1- \frac{e^c}{N_n(t_i-)} \text{ for some } i\ge 1\Big\}
\end{align*}
Given $N_n(t-)$ this event appears at time $t$ with rate
\[ \nu_{n,t}= \int_{[0,1-e^c/N_n(t-)]} p^{N_n(t-)} \frac{\Lambda(dp)}{p^2} \ .\]
Using the inequalities $p^b=(1-(1-p))^b\le e^{-(1-p)b} \le 1/((1-p)b)$ we get
\[ \nu_{n,t} \le \int_{[0,1-e^c/N_n(t-)]} e^{-(1-p)(N_n(t-)-2)} \, \Lambda(dp)\le \int_{[0,1-e^c/N_n(t-)]}\frac {e^2}{(1-p)N_n(t-)} \, \Lambda(dp) \ . \]
It follows
\begin{align*}
\mathbf E\Big[ \int_0^\infty \nu_{n,t} \, dt \Big] \le \mathbf E \Big[ \int_{[0,1]} \int_0^\infty \frac {e^2}{(1-p)N_n(t-)} I_{\{N_n(t-) \ge \lceil e^c/(1-p)\rceil \}} \, dt \, \Lambda(dp) \Big]
\end{align*}
 Lemma 14 of \cite{Ke} yields the estimate
\[\mathbf E \Big[  \int_0^\infty \frac {1}{N_n(t-)} I_{\{N_n(t-) \ge \lceil e^c/(1-p)\rceil  \}} \, dt\Big]  \le  c_1 \lceil e^c/(1-p)\rceil ^{-1}\le c_1 \frac{1-p}{e^c}\]
with some $c_1>0$, hence
\[ \mathbf E\Big[ \int_0^\infty \nu_{n,t} \, dt \Big] \le c_1 e^{2-c} \Lambda ([0,1])\ . \]
Therefore for $c$ sufficiently large
\[ \mathbf E\Big[ \int_0^\infty \nu_{n,t} \, dt \Big] \le \varepsilon/2 \ , \]
which implies
\[ \mathbf P\big(\Delta_n \le \log N_n(\tau_n-)-c\big) = 1-\exp \Big(- \mathbf E\Big[ \int_0^\infty \nu_{n,t} \, dt \Big]\Big) \le \varepsilon/2 \ .\]

Altogether we obtain
\[ \mathbf P\big(\sup_{t<\tau_n}| \log N_n(t)-Y_n(t)| \le r \, , \, \Delta_n > \log N_n(\tau_n-)-c\big) \ge 1-\varepsilon \ . \]
The event in the previous formula implies
\[ Y_n(\tau_n) = Y_n(\tau_n-)-\Delta_n < \log N_n(\tau_n-)+r -(\log N_n(\tau_n-) -c) = r+c \ ,\]
and the claim of the theorem follows with $\ell=r+c$. 
\end{proof}

\section{Proof of the main results}

\begin{proof}[Proof of Theorem \ref{LLN}]
Let us first assume that 
$\mu  < \infty$.
Then we have a coalescent with dust, and we may apply Proposition \ref{Prop5}.  Fix $\eta >0$. Note that on the event that $Y_n(\tau_n)<\ell$ the event $\tau_n < (1- \eta)\log n/\mu$ implies the inequality $T_\ell^{\log n} < (1-\eta) \log n/\mu$.  Thus in view of Proposition~\ref{Prop5}  there exists for any $\varepsilon >0$ an $\ell$ such that
\begin{align*}
\mathbf P( \tau_n < (1- \eta)\log n/\mu) \le  \mathbf P( T_\ell^{\log n} < (1-\eta) \log n/\mu)+ \varepsilon \ .
\end{align*}
Proposition \ref{Prop3} implies that the right-hand probability converges to 0 as $n \to \infty$. Letting $\varepsilon \to 0$ we obtain
\[ \lim_{n \to \infty} \mathbf P( \tau_n < (1- \eta)\log n/\mu) =0\ . \]

Also on the event $\sup_{t<\tau_n}| \log N_n(t)-Y_n(t)| \le \ell$, the event $\tau_n > (1+\eta)\log n/\mu $ implies $Y_n(t) \ge -\ell$ for all $t\le (1+\eta)\log n/\mu$, and consequently
\[ \mathbf P(\tau_n > (1+\eta)\log n/\mu) \le \mathbf P( T_{-\ell}^{\log n} > (1+\eta) \log n/\mu)+ \varepsilon \ . \]
Again the right-hand probability converges to zero in view of Proposition \ref{Prop3}, and we obtain
\[\lim_{n \to \infty} \mathbf P(\tau_n > (1+\eta)\log n/\mu) =0 \ . \]
Altogether our claim follows in the case $\mu < \infty$.

Now assume $\mu = \infty$.  If $\Lambda(\{0\})>0$, then the coalescent comes down from infinity and $\tau_n$ stays  bounded in probability. The same is true if $\Lambda(\{1\})>0$, thus we may assume that $\Lambda(\{0,1\})=0$.

For given $\varepsilon >0$ define the measure $\Lambda^\varepsilon$ by $\Lambda^\varepsilon(B):= \Lambda(B\cap [\varepsilon,1-\varepsilon])$. Obviously
\[ \mu^\varepsilon:=\int_0^1 \log \frac 1{1-p} \, \frac{\Lambda^\varepsilon(dp)}{p^2} < \infty\ . \]
Thus for the absorption times $\tau_n^\varepsilon$ of the $\Lambda^\varepsilon$-coalescent we have
\[  \frac {\tau^\varepsilon_n}{\log n} \to \frac 1{\mu^\varepsilon}\]
in probability as $n \to \infty$. Now we may couple the $\Lambda^\varepsilon$-coalescent in an obvious manner to the $\Lambda$-coalescent in such a way that  $N_n(t)\le N_n^\varepsilon(t)$ a.s. for all $t \ge 0$, in particular $\tau_n \le \tau_n^\varepsilon$. Hence it follows that 
\[ \mathbf P(\tau_n/\log n > 2 /\mu^\varepsilon) \to 0 \ . \]
Because of $\Lambda(\{0,1\})=0$ we have $\mu^\varepsilon \to \mu=\infty $ with $\varepsilon \to 0$, consequently
\[ \mathbf P(\tau_n/\log n > \eta) \to 0\]
for all $\eta >0$. This is our claim.
\end{proof}

\begin{proof}[Proof of Theorem \ref{CLT}]
Because of the condition $\mu<\infty$ we again may apply Proposition \ref{Prop5}.
We follow the same line as in the previous proof:
For $\varepsilon >0$ there exists an $\ell$ such that for all $a\in \mathbb R$
\begin{align*}
\mathbf P( \tau_n < b_n+ a \sqrt n) \le  \mathbf P( T_\ell^{\log n} < b_n+ a \sqrt n)+ \varepsilon
\end{align*}
and
\begin{align*}
\mathbf P(\tau_n > b_n+a\sqrt n) \le \mathbf P( T_{-\ell}^{\log n} > b_n+a\sqrt n)+ \varepsilon 
\end{align*}
Now apply Proposition \ref{Prop4} and let $\varepsilon \to 0$.
\end{proof}

\begin{proof}[Proof of Proposition \ref{Prop2}]
(i) Let us first assume \eqref{condi}. Because of $1-(1-p)^{1/r} \le \min (p/r,1)$ for  $0<r<1$ we have for $\alpha > 0$ 
\begin{align} f\big(\log \tfrac 1r\big) \le \int_0^{r^\alpha} \frac{\Lambda (dp)}p +  r\int_{r^\alpha}^1 \frac{ \Lambda(dp)}{p^2} \le \int_0^{r^\alpha} \frac{\Lambda (dp)}p + r^{1-\alpha }\int_0^1 \frac{\Lambda(dp)}{p}  \ .
\label{estimate1}
\end{align}
Also,  because of $1-(1-p)^{1/r} \ge 1- e^{-p/r} \ge  e^{-p/r}p/r$, it follows for $\beta >0$ that 
\begin{align} f\big(\log \tfrac 1r\big) \ge e^{-r^{\beta -1}}\int_0^{r^\beta} \frac{\Lambda (dp)}p\ . 
\label{estimate2}
\end{align}
Together with  \eqref{condi}   these two estimates yield for $\alpha < 1 < \beta$  
\[   c \beta^{-1/2} \le \liminf_{r \to 0} f\big(\log \tfrac 1r\big) \sqrt{\log \tfrac 1r} \le \limsup_{r \to 0} f\big(\log \tfrac 1r\big) \sqrt{\log \tfrac 1r} \le c \alpha^{-1/2}  \ .\]
Letting $\alpha, \beta \to 1$ we arrive at $f(y)= (c+o(1))/\sqrt y$ as $y \to \infty$ and consequently
\[ \int_0^{\log n} f(y) \, dy = (c+o(1))2\sqrt {\log n}   \]
as $n \to \infty$. 

Now, because of 
\begin{align*} 
 \frac{1}{\mu-f(y)}= \frac 1\mu + \frac {f(y)}{\mu(\mu-f(y))} 
\end{align*}
and $f(y)=o(1)$ as $y\to \infty$, we have 
\begin{align}
\label{formula} \int_\kappa^z \frac{dy}{\mu-f(y)} = \frac z\mu + \frac{1+o(1)}{\mu^2} \int_0^z f(y)\, dy +O(1)
\end{align}
as $z\to \infty$, and consequently, as claimed,
\begin{align*} 
b_n  = \frac{\log n}\mu + \frac{2c+o(1)}{\mu^2} \, \sqrt{\log n} \ .
\end{align*}

(ii) Now suppose that \eqref{condi2} is satisfied. Then in view of \eqref{formula} with $z=\log n$ it follows that
\[ \int_0^{\log n} f(y)\, dy = (2c+o(1)) \sqrt {\log n} \]
as $n \to \infty$, or equivalently
\[ \int_0^{z} f(y)\, dy = (2c+o(1)) \sqrt z \]
for $z \to \infty$. This implies that 
$f(z)= (c+o(1))/ \sqrt z$ as $z\to \infty$. For $c=0$ this claim follows because $f$ is decreasing, which entails
\[ zf(z) \le \int_0^z f(y)\, dy = o(\sqrt z) \ . \]
For $c>0$ we use the estimate
\[ \frac 1{\eta \sqrt z}\int_z^{(1+\eta) z} f(y)\, dy \le   \sqrt z f(z) \le \frac 1{\eta \sqrt z} \int_{(1-\eta)z}^z f(y) \, dy \]
 with $\eta >0$. Taking the limit $z\to \infty$ and then $\eta \to 0$ yields $f(z)= (c+o(1))/ \sqrt z$. Now, similar as in part (i) we get  from \eqref{estimate1} and \eqref{estimate2}  
\[c\sqrt \alpha \le \liminf_{r \to 0}\sqrt {\log \tfrac 1r}\int_{[0,r]} \frac{\Lambda(dp)}p\le  \limsup_{r \to 0}\sqrt {\log \tfrac 1r}\int_{[0,r]} \frac{\Lambda(dp)}p \le  c\sqrt \beta\ .\]
With $\alpha, \beta \to 1$ we arrive at \eqref{condi}.
\end{proof}

\paragraph{Acknowledgement.} It is our pleasure to dedicate this work to Peter Jagers.

\end{document}